\documentclass[12pt]{amsart}
\usepackage{amsmath,amssymb,amsthm,mathtools}
\usepackage[margin=1.15in]{geometry}
\usepackage{microtype}
\usepackage{thmtools}
\usepackage[hypertexnames=false]{hyperref}
\hypersetup{
  colorlinks=true,
  linkcolor=blue,
  citecolor=blue,
  urlcolor=blue
}
\usepackage[nameinlink,capitalise]{cleveref}

\newtheorem{theorem}{Theorem}[section]
\newtheorem{lemma}[theorem]{Lemma}
\newtheorem{proposition}[theorem]{Proposition}
\newtheorem{corollary}[theorem]{Corollary}
\theoremstyle{definition}
\newtheorem*{definition}{Definition}
\newtheorem*{remark}{Remark}
\newtheorem*{example}{Example}
\newtheorem{case}{Case}[theorem]
\Crefname{case}{Case}{Cases}

\newcommand{\ZZ}{\mathbb Z}
\newcommand{\FF}{\mathbb F}
\newcommand{\PB}{P_B}
\newcommand{\TB}{\mathcal{T}_B}

\title[Four-digit Kaprekar dynamics in odd bases]{Four-digit Kaprekar dynamics in odd bases}
\author{Evan Chen}
\author{Ken Ono}
\author{Richard E. Schwartz}
\author{Dinesh S. Thakur}

\address{Axiom Math, 124 University Avenue, Palo Alto, CA 94301}
\email{evan@axiommath.ai}
\email{ken@axiommath.ai}

\address{Department of Mathematics, Brown University, Providence, RI 02912}
\email{res@math.brown.edu}

\address{Department of Mathematics, University of Rochester, Rochester, NY 14627}
\email{dinesh.thakur@rochester.edu}

\thanks{Richard E. Schwartz's research is supported by NSF grant  DMS-2505281.}

\begin{document}

\begin{abstract}
  Start with four digits, arrange them in both descending and ascending order, subtract, and repeat.
  This simple process is known as the Kaprekar routine,
  famous in base ten for sending every nonconstant four-digit string to $6174$.
  We show that in every odd base $B>3$, the four-digit Kaprekar map has an unexpectedly rigid structure.
  After at most three iterations,
  every nonconstant orbit enters an explicit triangular region $\mathcal{T}_B$,
  and on this region the map is conjugate to projective doubling:
  \[ \{[r],[s]\}\longmapsto \{[2r],[2s]\}.  \]
  This gives a complete finite description of all nonconstant terminal cycles,
  including an explicit formula for their lengths and counts.
  In particular, the longest terminal cycle has length at most $(B-1)/2$,
  and equality can occur only when $B$ is prime.
  For primes $p>5$, equality occurs precisely when the least positive $m$ with
  $2^m\equiv\pm1\pmod p$ is $m=(p-1)/2$.
  The results proved here were first formulated by Schwartz and Thakur.
  As a test case for AI-assisted formal mathematics,
  AxiomProver produced Lean/mathlib formalizations of these results.
\end{abstract}



\maketitle
\enlargethispage{2pt}

\section{Introduction and main results}
\subsection{Background}
Kaprekar's four-digit routine looks, at first, like a numerical magic trick.
Take four digits, arrange them in decreasing order, subtract the same digits in increasing order, and repeat.
In base ten, the trick is famously dramatic: the process is drawn to the constant $6174$,
which is a fixed point since
\[ 7641-1467=6174.  \]
For example, starting with the number $2808$, we have the process
\[ 2808\longmapsto 8532\longmapsto 6174\longmapsto 6174.  \]
The surprise is not merely that a fixed point appears,
but that a rule made from sorting digits is secretly governed by modular arithmetic.
This paper explains that hidden arithmetic for four digits in every odd base.

Kaprekar introduced this routine in his study of decimal digit phenomena \cite{Kaprekar1949,Kaprekar1955}.
In the classical four-digit decimal case, every nonconstant digit string reaches the fixed point $6174$.
For 3 decimal digits, it is $495$.
The decimal case, however, is only one member of a larger family.
In other bases the four-digit routine need not have a single attracting fixed point.
Its terminal behavior may consist of several cycles, and their lengths depend on the arithmetic of the base.

The generalized Kaprekar problem has been studied from several points of view. Moreover, the fixed points of Kaprekar mappings define the integer sequence \href{https://oeis.org/A099009}{OEIS A099009}.
Jordan studied self-producing digit sequences \cite{Jordan}.
Hasse and Prichett determined the four-digit Kaprekar constants (i.e.,
fixed points for general bases) \cite{HassePrichett}.
Prichett analyzed five-digit terminating cycles \cite{Prichett}.
Thakur developed a broader framework for Kaprekar phenomena in arbitrary bases
and digit lengths \cite{Thakur}.
Devlin and Zeng studied maximum distances to the four-digit fixed point in the
bases where such a fixed point exists \cite{DevlinZeng}.

More recently, Yamagami and Matsui developed a detailed theory of $3$-adic
Kaprekar loops and obtained formulas for Kaprekar fixed points
\cite{YamagamiMatsuiLoops,YamagamiMatsuiConstants}.
Kay and Downes-Ward extended fixed-point and cycle questions to odd bases using
Kaprekar indices and subgroup/coset methods, and their sequel treats even bases,
where the classification is more intricate \cite{KayDownesWardOdd,KayDownesWardEven}.
Our purpose here is more explicit.
In the four-digit odd-base case, the right coordinates reveal a particularly simple mechanism:
after a bounded pre-periodic segment, the Kaprekar map is just projective doubling.
The resulting proof is short, but the simplification is quite striking,
and it gives a transparent route to the cycle-length consequences below while
keeping the digit dynamics visible throughout.

\subsection{Main classification result}
Given a $4$-digit positive integer in an odd base $B$,
let $a_1 \ge a_2 \ge a_3 \ge a_4$ denote its digits in non-increasing order.
As noted by Hasse and Prichett in \cite{HassePrichett},
the next Kaprekar iterate depends only on the two differences
\begin{equation}\tag{D}\label{eq:intro-difference-coordinates}
  d_1=a_1-a_4
  \qquad\text{and}\qquad
  d_2=a_2-a_3.
\end{equation}
Thus $d_1$ records the outer difference and $d_2$ records the inner difference.
Accordingly, we define a natural state space
\begin{equation}
  \tag{X} \label{eq:X_B}
  X_B \coloneq \{(d_1,d_2):1\leq d_1<B,\ 0\leq d_2\leq d_1\}.
\end{equation}
The Kaprekar routine thus induces a map $K_B \colon X_B\to X_B$,
sending each difference pair to the next difference pair.
Note here we have excluded the trivial case $a_1 = a_2 = a_3 = a_4$ by requiring $d_1 \neq 0$.

In \Cref{sec:transient} we will prove that, after a short pre-periodic segment,
these two differences become positive, unequal, and odd; that is, $(d_1,d_2)$
will lie in the stable odd region defined by
\begin{equation}\tag{S}\label{eq:stable-chamber}
  \TB \coloneq \{(d_1,d_2)\in X_B: d_1>d_2>0
    \text{ and } d_1\equiv d_2\equiv1\pmod2\}.
\end{equation}
We will also describe how $K_B$ acts on $\TB$.
To give this description, rather than working with $d_1$ and $d_2$ themselves,
we'll want to work with their half-sum and half-difference:
\begin{equation}\tag{R}\label{eq:intro-rs-coordinates}
  r=\frac{d_1+d_2}{2}
  \qquad\text{and}\qquad
  s=\frac{d_1-d_2}{2}.
\end{equation}
(Note that $r > s > 0$, and that $r$ and $s$ are integers because $d_1$ and $d_2$ are odd.)

For pairs $(r,s)$ we will pass to \emph{projective residue classes modulo $B$},
where each residue is identified with its negative.
Introduce the notation
\[ \PB \coloneq \left( (\ZZ/B\ZZ) \setminus \{0\} \right) / \{\pm1\}  \]
for the nonzero projective residue classes.
For nonzero $x \in \ZZ/B\ZZ$, let $[x] \coloneq \{x,-x\} \in \PB$
denote the corresponding residue class in $\PB$.
Finally, let $\binom{\PB}{2}$ denote the (unordered) two-element subsets of $\PB$.


We can now state the main structural result.
\begin{theorem}[Structural classification]\label{thm:structural}
  Let $B>3$ be odd, and let $K_B \colon X_B\to X_B$ be the four-digit
  Kaprekar map in difference coordinates.
  Then $\TB$, as defined in \eqref{eq:stable-chamber}, is forward-invariant and
  \[ K_B^3(X_B)\subseteq \TB. \]
  Moreover, the restricted map $K_B \colon \TB\to \TB$ is conjugate to the doubling map
  \[ \binom{\PB}{2} \to \binom{\PB}{2} \quad \text{ by }\quad
  \{[r],[s]\}\longmapsto \{[2r],[2s]\} \]
  under the transformation defined in \eqref{eq:intro-rs-coordinates}.
\end{theorem}

Here, in the context of finite dynamical systems,
we explicate the term ``conjugate'' as follows.
Let $f \colon X\to X$ and $g \colon Y\to Y$ be maps.
We say that $f$ and $g$ are \emph{conjugate} under the bijection
$\Phi \colon X\to Y$ if
\[ \Phi(f(x))=g(\Phi(x)) \qquad\text{for all }x\in X.  \]
Thus $\Phi$ is a relabeling of the states which turns one system into the other.
In the particular case of \Cref{thm:structural},
the map $\Phi \colon \TB \to \binom{\PB}{2}$
is given by \eqref{eq:intro-rs-coordinates},
and so \Cref{thm:structural} makes an implicit promise that
\eqref{eq:intro-rs-coordinates} is bijective.

We can thus describe \Cref{thm:structural} in English as follows:
After at most three iterations,
every orbit enters our smaller stable region $\TB$.
And on this region $\TB$, the sorting-and-subtracting rule is no longer mysterious!
After the change of variables in \eqref{eq:intro-rs-coordinates},
it is exactly multiplication by $2$, with signs ignored,
on unordered pairs of nonzero projective residue classes.

\begin{remark}
  Here the case $B = 3$ needs to be excluded because $\TB$ is empty when $B = 3$.
\end{remark}

\subsection{Cycle lengths}
It is clear that for odd $B$ the doubling map
\[ \binom{\PB}{2} \to \binom{\PB}{2} \text{ by } \{[r],[s]\}\longmapsto \{[2r],[2s]\} \]
is bijective.
Since conjugate systems have the same cycle lengths,
we get the following corollaries as the main consequences.
They are stated separately so that the dynamical statement and the arithmetic
consequences are easy to distinguish.

\begin{corollary}[Cycle-length bound]\label{cor:length}
  If $B>3$ is odd and $c_{\max}(B)$ is the largest terminal-cycle length for the
  four-digit Kaprekar map in base $B$, then we have
  \[ c_{\max}(B)\le \frac{B-1}{2}.  \]
  Moreover, equality holds if and only if $B \ge 7$, $B$ is prime,
  and the least positive solution to $2^m \equiv \pm 1 \pmod B$ is $m = \frac{B-1}{2}$.
\end{corollary}

\begin{remark}
  In the edge case $B = 5$, one has $c_{\max}(5)=1<2=(B-1)/2$ instead
  (see \cite[Lemma 1]{HassePrichett}, or \cite[\S4 and Appendix B]{KayDownesWardOdd}).
\end{remark}

Thus \cref{cor:length} gives a universal ceiling
and describes exactly when it is achieved.
We also have the following result:

\begin{corollary}[Prime bases and maximal cycles]\label{cor:primecriterion}
  If equality holds in \Cref{cor:length} (so that $B$ is prime),
  then the number of terminal cycles of length $c_{\max}(B)$ is
  \[ \left\lfloor\frac{c_{\max}(B)-1}{2}\right\rfloor.  \]
\end{corollary}

\begin{remark}
  Our results are a four-digit complement to the broader odd-base theory.
  Yamagami and Matsui give a detailed theory in base $3$ \cite{YamagamiMatsuiLoops},
  and Kay and Downes-Ward study fixed points and cycles in odd bases using Kaprekar indices,
  with regular cycles organized through subgroup and coset data \cite{KayDownesWardOdd}.
  Their framework treats arbitrary digit-counts and is especially well suited to regular cycles.
  The present note fixes the digit-count at four.
  In this small case, the dynamics are more directly described by the outer and inner digit differences,
  and after a bounded pre-periodic segment these two differences give a stable region on
  which the map becomes projective doubling.
  Therefore, the contribution here is a low-dimensional model that makes the
  four-digit odd-base dynamics and their cycle counts explicit.
\end{remark}

\subsection*{Provenance and formal verification}
The structural conjectures leading to Theorem~\ref{thm:structural}
and its cycle-length consequences were first formulated by Schwartz and
Thakur and communicated to Ono. AxiomProver was then used as an
experimental AI-assisted formal mathematics system to produce
Lean/mathlib formalizations of the resulting statements. The present
paper gives a human-readable account of the mathematics. The formal
files, the Lean environment, and the verification protocol are described
in Section~\ref{sec:AI}. AxiomProver is credited as a tool rather than
as an author, and the human authors take responsibility for the
mathematical statements, exposition, attribution, and correctness of the
paper.

\subsection*{Examples of the main results}
\label{ex:main-examples}
Here are a few examples that illustrate these results.
The cycles are controlled not by decimal accidents,
but by the order of $2$ modulo the base, with signs ignored.
We start each example with ordinary four-digit strings,
and only then pass to the difference and projective coordinates.

\begin{example}[The case $B = 7$; {\cite[\S5]{KayDownesWardOdd}}]
  The sorted four-digit string $3100$ has
  \[
        d_1=3-0=3
        \qquad\text{and}\qquad
        d_2=1-0=1.
  \]
  Thus
  \[
        r=\frac{3+1}{2}=2
        \qquad\text{and}\qquad
        s=\frac{3-1}{2}=1.
  \]
  Writing each iterate as a sorted digit string, the orbit of $3100$ is
  \[
        3100\longmapsto 5430\longmapsto 5520\longmapsto 5322\longmapsto 5430\longmapsto\dots,
  \]
  which in difference coordinates is the terminal cycle
  \[
        (3,1)\longmapsto (5,1)\longmapsto (5,3)\longmapsto (3,1).
  \]
  Here $(B-1)/2=3$, and $2$ has projective order $3$ modulo $7$,
  so the largest four-digit Kaprekar cycles have length $3$.
  This agrees with the count $\lfloor(3-1)/2\rfloor=1$.
\end{example}

\begin{example}[The case $B = 11$]
  The sorted four-digit string $5100$ has $(d_1,d_2)=(5,1)$, and hence
  \[
        r=3
        \qquad\text{and}\qquad
        s=2.
  \]
  Writing each iterate as a sorted digit string, the orbit of $5100$ is
  \[
        5100\longmapsto 9650\longmapsto 9920\longmapsto 9632\longmapsto 7742\longmapsto 7652\longmapsto 9650\longmapsto\dots.
  \]
  In difference coordinates this is the terminal cycle
  \[
        (5,1)\to(9,1)\to(9,7)\to(7,3)\to(5,3)\to(5,1).
  \]
  Since $2^5\equiv -1\pmod {11}$ and no smaller positive power of $2$ is congruent to $\pm1$,
  we have $c_{\max}(11)=5$.  There are
  \[
        \left\lfloor\frac{5-1}{2}\right\rfloor=2
  \]
  cycles of length five; the other one is
  \[
        (3,1)\to(9,5)\to(7,1)\to(9,3)\to(7,5)\to(3,1).
  \]
\end{example}

\begin{example}[The case $B = 9$]
  The composite base $B=9$ shows why the inequality is usually strict.
  Starting from the sorted string $5100$, we have $(d_1,d_2)=(5,1)$ and $(r,s)=(3,2)$.
  The terminal behavior in difference coordinates is
  \[
        (5,1)\to(7,1)\to(7,5)\to(5,1),
  \]
  with the second cycle
  \[
        (3,1)\to(7,3)\to(5,3)\to(3,1).
  \]
  Here $(B-1)/2=4$, but $2^3\equiv -1\pmod 9$, so the terminal cycles have length $3$, not $4$.
\end{example}

\begin{example}[The case $B = 17$]
  Finally, a prime base need not give equality.
  For $B=17$, the sorted four-digit string $5100$ again has $(d_1,d_2)=(5,1)$ and $(r,s)=(3,2)$,
  but $2^4\equiv -1\pmod {17}$.  Thus the largest terminal cycles have length $4$, not $(17-1)/2=8$.
\end{example}

\medskip
The rest of the paper is organized as follows.
\Cref{sec:transient} introduces the difference coordinates and proves that every orbit
reaches the stable odd region after at most three steps.
\Cref{sec:projective} identifies the hidden projective coordinates and proves that, on this region,
one Kaprekar step is just doubling in those coordinates.
\Cref{sec:cycle} translates the projective model into explicit cycle-count formulas,
and then deduces the general cycle-length bound and the prime-base obstruction to equality.
\Cref{sec:prime} treats prime bases in detail and counts the largest cycles when the bound is attained. Section~\ref{sec:AI} describes the accompanying Lean/mathlib
formalization, including the AxiomProver protocol, the formal files,
and the verification environment.

\section*{Acknowledgments}
\noindent
The authors thank Rub\'en Ballester and Andrew Granville for helpful comments regarding this project,
and Granville in particular for suggesting promising directions for future work that generalize these results.

\section{Difference coordinates for four digits}
\label{sec:transient}
We now turn from the statement of the result to the mechanism behind it.
The first step is to reduce the digit operation to a map on two integers.
This reduction is useful because the middle information in a sorted four-digit
string cancels during subtraction.

\subsection{Difference coordinates}
Fix an odd base $B>3$.  Let
\[
        \overrightarrow N=[a_1,a_2,a_3,a_4]
        \qquad\text{with}\qquad
        a_1\ge a_2\ge a_3\ge a_4
\]
be the decreasing rearrangement of a four-digit string.
We ignore constant strings, since they map to zero.
As in \eqref{eq:intro-difference-coordinates}, let $d_1 \coloneq a_1 - a_4$
and $d_2 \coloneq a_2 - a_3$ so that
\[
        1\le d_1<B
        \qquad\text{and}\qquad
        0\le d_2\le d_1.
\]
Then Kaprekar subtraction depends only on $(d_1,d_2)$, as in \cite{HassePrichett}.
Indeed,
\[
\begin{aligned}
  &(a_1B^3+a_2B^2+a_3B+a_4)
  -(a_4B^3+a_3B^2+a_2B+a_1)       \\
  &\hspace{1cm}=d_1(B^3-1)+d_2(B^2-B).
\end{aligned}
\]
Thus the future orbit after one step is determined by the pair $(d_1,d_2)$.

If $d_2>0$, ordinary base-$B$ subtraction gives the unsorted output digits
\begin{equation}\label{eq:d2positive}
        [d_1,\ d_2-1,\ B-d_2-1,\ B-d_1].
\end{equation}
If $d_2=0$, the subtraction has one fewer borrow, and the unsorted output digits are
\begin{equation}\label{eq:d2zero}
        [d_1-1,\ B-1,\ B-1,\ B-d_1].
\end{equation}
Let $K_B$ denote the induced map on $X_B$ as defined in \eqref{eq:X_B}.
To compute $K_B(d_1,d_2)$, one sorts the four entries in \eqref{eq:d2positive}
or \eqref{eq:d2zero} in decreasing order and then takes the outer and inner differences.
This passage to difference coordinates does not lose any eventual cycle lengths.
Indeed, a digit-level orbit determines a difference-coordinate orbit,
and conversely a periodic orbit in difference coordinates lifts to a digit-level
periodic orbit by the subtraction formulas \eqref{eq:d2positive} and \eqref{eq:d2zero}.
If two digit strings in such a lifted orbit were equal, their difference pairs would also be equal,
so the period is the same.  Thus it is enough to count terminal cycles for $K_B$ on $X_B$.

\subsection{Escaping the pre-period}
The rest of this section is a cleanup step.
We show that the boundary phenomena disappear after a bounded number of iterations,
leaving a simple stable region on which the projective model will live.

\begin{lemma}[Pre-period into the stable odd region]\label{lem:transient}
  For every odd $B>3$,
  \[ K_B^3(X_B)\subseteq \TB.  \]
  Moreover $\TB$ is forward invariant under $K_B$.
\end{lemma}

The proof of \Cref{lem:transient} will be divided into a few cases
resolved via explicit calculation of $K_B$.
This calculation also appears in \cite[equations (1.1)--(1.3)]{HassePrichett},
and \cite[\S2]{DevlinZeng}, although our organization of the cases is slightly different.

\begin{case}
  \Cref{lem:transient} holds in the case $d_1 > d_2 > 0$.
  In fact:
  \begin{itemize}
    \item If $d_1 > d_2 > 0$ and $d_1 \equiv d_2 \pmod 2$, then $K_B(d_1, d_2) \in \TB$.
    \item If $d_1 > d_2 > 0$ but $d_1 \not\equiv d_2 \pmod 2$, then $K_B^2(d_1, d_2) \in \TB$.
  \end{itemize}
  \label{case:d1_gt_d2_gt_0}
\end{case}
\begin{proof}
  [Proof for \Cref{case:d1_gt_d2_gt_0}]
  Abbreviate the four entries in \eqref{eq:d2positive} by
  \[
          A \coloneq d_1
          \qquad\text{and}\qquad
          C \coloneq d_2-1
          \qquad\text{and}\qquad
          D \coloneq B-d_2-1
          \qquad\text{and}\qquad
          E \coloneq B-d_1.
  \]
  We always have $A>C$ and $D\ge E$; also $A \neq E$ and $C \neq D$.
  We note the case $D > A > C > E$ is also impossible,
  because it would require both $d_1+d_2<B-1$ and $d_1+d_2>B+1$.
  Hence, we must be in one of the five situations listed in \Cref{tab:ACDE}.
  Of course, only rows whose displayed differences are nonnegative can actually occur.

  \begin{table}[ht]
    \[
    \begin{array}{ccc}
      \hline
      \text{Ordering} & K_B(d_1,d_2) & \text{In $\TB$?} \\
      \hline
      A \ge D \ge E \ge C & (d_1-d_2+1,\ d_1-d_2-1) & \text{Iff } d_1 \equiv d_2 \bmod 2 \\
      D \ge E > A > C & (B-2d_2,\ B-2d_1) & \text{Always} \\
      A \ge D > C \ge E & (2d_1-B,\ B-2d_2) & \text{Always} \\
      D \ge A > E \ge C & (B-2d_2,\ 2d_1-B) & \text{Always} \\
      A > C > D \ge E & (2d_1-B,\ 2d_2-B) & \text{Always} \\
      \hline
    \end{array}
    \]
    \caption{Possible results of $K_B(d_1, d_2)$ when $d_1 > d_2 > 0$.}
    \label{tab:ACDE}
  \end{table}

  We assert the last four rows of \Cref{tab:ACDE} always lie in $\TB$.
  Indeed, it is clear the entries are odd.
  Moreover, they are easily seen to be unequal;
  in the third row we have $C \ge E \implies d_1 + d_2 \ge B+1$
  while in the fourth row we have $D \ge A \implies d_1 + d_2 \le B-1$.
  Hence in the third and fourth row $d_1 + d_2 \neq B$ always holds
  and the two entries are unequal.

  It remains to only deal with the first row,
  which lies in $\TB$ exactly when $d_1 \equiv d_2 \pmod 2$.
  It suffices to show that
  \begin{equation}
    K_B(h+1, h-1) \in \TB \quad\text{for each odd}\quad h \ge 1.
    \label{eq:h_pm_1}
  \end{equation}
  Indeed \eqref{eq:h_pm_1} can be checked directly:
  \begin{itemize}
    \item If $h \neq 1$, we apply \Cref{tab:ACDE} a second time to the new pair $(h+1, h-1)$.
      For $(h+1, h-1)$, all five rows of \Cref{tab:ACDE} are in $\TB$ and we are done.
    \item If $h = 1$ then in fact $(h+1, h-1) = (2, 0)$.
      Although \Cref{tab:ACDE} does not apply, we can directly compute
      \[ K_B(2, 0) = (B-2, 1) \in \TB. \qedhere \]
  \end{itemize}
\end{proof}

\begin{case}
  We have $K_B^2(d,d) \in \TB$ for any $d > 0$.
  \label{case:d_d}
\end{case}
\begin{proof}
  [Proof for \Cref{case:d_d}]
  Formula \eqref{eq:d2positive} becomes \[ [d,d-1,B-d-1,B-d].  \]
  Then \[ K_B(d,d) = (h+1, h-1) \qquad\text{where}\qquad h \coloneq |B-2d|. \]
  This is the same as \eqref{eq:h_pm_1} from the earlier case,
  so we get $K_B(h+1,h-1) \in \TB$ as desired.
\end{proof}

\begin{case}
  We have $K_B^3(d_1, 0) \in \TB$ for any $d_1 > 0$.
  \label{case:d1_0}
\end{case}
\begin{proof}
  [Proof for \Cref{case:d1_0}]
  Put
  \[ u=d_1-1 \qquad\text{and}\qquad v=B-d_1.  \]
  Sorting \eqref{eq:d2zero} gives
  \begin{equation}\label{eq:zeroformula}
    K_B(d_1,0)=\bigl(B-1-\min(u,v),\ B-1-\max(u,v)\bigr).
  \end{equation}
  If the second coordinate in \eqref{eq:zeroformula} is zero, then $d_1=1$ and
  \[ (1,0)\mapsto (B-1,0)\mapsto (B-2,1)\in \TB.  \]
  If the image in \eqref{eq:zeroformula} has positive unequal coordinates,
  then \Cref{case:d1_gt_d2_gt_0} applies.
  If the image has equal coordinates, then \Cref{case:d_d} applies.
  Hence every state with $d_2=0$ reaches $\TB$ within at most three steps.
\end{proof}

We can now deduce the main lemma.

\begin{proof}
  [Proof of \Cref{lem:transient}]
  We note that $K_B(\TB) \subseteq \TB$ follows from the first half of \Cref{case:d1_gt_d2_gt_0},
  while $K_B^3(X_B) \subseteq \TB$ follows by noting that
  \Cref{case:d1_gt_d2_gt_0,case:d_d,case:d1_0} cover all possibilities.
\end{proof}

\begin{example}[The pre-periodic segment in base $11$]\label{ex:transient}
  Let $B = 11$.
  An example for \Cref{case:d_d} is
  \[ (3,3) \longmapsto (6,4) \longmapsto (3,1) \in \mathcal{T}_{11}.  \]
  An example for \Cref{case:d1_0} is
  \[ (6,0) \longmapsto (5,5)\longmapsto (2,0) \longmapsto (9, 1) \in \mathcal{T}_{11}.  \]
  These two computations illustrate why a short pre-periodic segment
  is needed before the simple projective model appears.
\end{example}

\begin{remark}
  \Cref{lem:transient} is analogous to Hasse-Prichett's \cite[Lemma 3.1]{HassePrichett},
  which defines the analog of $\TB$ for $B = 5 \cdot 2^n$.
  In our case of odd $B$ we are able to enforce a uniform bound of $3$
  on the number of applications of $K_B$ needed.
  But when $B = 5 \cdot 2^n$ the number of steps needed to reach a cycle
  can grow linearly in $n$ (see Devlin-Zeng \cite[Theorem 1.1]{DevlinZeng}).
\end{remark}

\section{The projective model}
\label{sec:projective}
The previous section deals with the pre-periodic segment.  We now reveal the hidden coordinates on the stable region.
They are not the differences themselves, but the half-sum and half-difference of the two differences,
viewed modulo sign.

Recall that $\PB=((\ZZ/B\ZZ)\setminus\{0\})/\{\pm1\}$.
Define
\begin{equation}\tag{P}\label{eq:Phi}
  \begin{aligned}
    \Phi \colon \TB &\to \binom{\PB}{2} \\
    \text{by } (d_1,d_2) &\mapsto
        \left\{\left[\frac{d_1+d_2}{2}\right],
        \left[\frac{d_1-d_2}{2}\right]\right\}.
  \end{aligned}
\end{equation}
The two classes are nonzero and distinct, because $d_1>d_2>0$ and $d_1<B$.
We first prove this is a bijection.

\begin{lemma}[Bijection with projective pairs]\label{lem:bijection}
  The map $\Phi \colon \TB \to \binom{\PB}{2}$ is a bijection.
\end{lemma}

\begin{proof}
  Define the intermediate set
  \[ \mathcal{R} \coloneq \{(r,s) \mid r > s > 0,\ r+s < B,\ r \not\equiv s \bmod 2\}. \]
  Recall the notation from \eqref{eq:intro-rs-coordinates}.
  We first contend that the map
  \begin{equation}
    \TB \to \mathcal{R} \qquad\text{by}\qquad (d_1, d_2) \mapsto (r, s)
    = \left( \frac{d_1+d_2}{2}, \frac{d_1-d_2}{2} \right)
    \label{eq:Phi1}
  \end{equation}
  is a (well-defined) bijection, with inverse given by $d_1 = r+s$ and $d_2 = r-s$.
  Indeed, the conditions on $\TB$ and $\mathcal{R}$ line up exactly:
  \begin{align*}
    d_1, d_2 \text{ odd} &\iff r \pm s \text{ odd} \\
    d_1 > d_2 &\iff s > 0 \\
    d_2 > 0 &\iff r > s \\
    B > d_1 &\iff r+s < B.
  \end{align*}
  This establishes \eqref{eq:Phi1} is indeed a bijection.

  So, it remains to prove that the projection
  \begin{equation}
    \mathcal{R} \to \binom{\PB}{2} \qquad\text{by}\qquad (r,s) \mapsto \{[r], [s]\}
    \label{eq:Phi2}
  \end{equation}
  is also a bijection.
  To that end, suppose $\{[u], [v]\} \in \binom{\PB}{2}$ is an arbitrary
  pair of distinct projective classes, with $0 < u < B$ and $0 < v < B$.
  Since $[u] \neq [v]$ we also have $u \neq v$ and $u+v \neq B$.

  The main point is that among the four ordered pairs
  $(u,v)$, $(u,B-v)$, $(B-u,v)$, $(B-u,B-v)$
  it is easy to see exactly one ordered pair has odd sum less than $B$.
  (Indeed, exactly two pairs have odd sum, and the sum of all four numbers is $2B$.)

  So WLOG assume (by renaming $u$ to $B-u$ or $v$ to $B-v$ if needed)
  that $u+v < B$ is odd.
  In that case, $r = \max(u,v)$ and $s = \min(u,v)$
  is the unique pair in $\mathcal{R}$ with image $\{[u], [v]\}$.
  Hence \eqref{eq:Phi2} is also a bijection.

  \Cref{lem:bijection} now follows by composing \eqref{eq:Phi1} and \eqref{eq:Phi2}.
\end{proof}

\begin{example}[The bijection in base $11$]\label{ex:bijection}
  In base $11$, consider the two classes $[2]=[9]$ and $[5]=[6]$.
  The unique $(r,s)$ pair is $r=5$ and $s=2$, which gives
  $d_1 = 7$ and $d_2 = 3$.
  Hence the bijection in this case is
  \[ \Phi(7,3) = \{[5], [2]\}. \]
\end{example}

The next proposition is the heart of the paper.
It says that, after the pre-periodic segment,
the digit dynamics are exactly multiplication by $2$ on the two projective coordinates.

\begin{proposition}[Projective doubling]\label{prop:conjugacy}
  For every $(d_1,d_2)\in \TB$,
  \begin{equation}\label{eq:conjugacy}
        \Phi\bigl(K_B(d_1,d_2)\bigr)
        =
        \left\{\left[2\cdot\frac{d_1+d_2}{2}\right],
        \left[2\cdot\frac{d_1-d_2}{2}\right]\right\}.
  \end{equation}
  Equivalently, under $\Phi$, the Kaprekar map on $\TB$ is conjugate to
  \[
        \{[r],[s]\}\longmapsto\{[2r],[2s]\}
  \]
  on unordered projective pairs.
\end{proposition}

\begin{proof}
  We reuse the earlier \Cref{tab:ACDE} from the proof of \Cref{case:d1_gt_d2_gt_0}
  (which is applicable because we are in $\TB$),
  but rewrite it in $(r,s)$-coordinates.
  That is, we calculate $K_B(d_1, d_2) = K_B(r+s, r-s)$ in terms of $r$ and $s$,
  then compute the half-sum and half-difference of the new pair, modulo $B$.
  This transforms \Cref{tab:ACDE} into \Cref{tab:doubling} below.

  \begin{table}[ht]
    \[
      \begin{array}{cccc}
        \hline
        \text{Ordering} & K_B(r+s,r-s)
        & \text{Half-sum mod $B$} & \text{Half-diff mod $B$} \\
        \hline
        A \ge D \ge E \ge C & (2s+1,\ 2s-1) & 2s & 1 \\
        D \ge E > A > C & (B-2r+2s,\ B-2r-2s) & -2r & 2s \\
        A \ge D > C \ge E & (2r+2s-B,\ B-2r+2s) & 2s & 2r \\
        D \ge A > E \ge C & (B-2r+2s,\ 2r+2s-B) & 2s & -2r \\
        A > C > D \ge E & (2r+2s-B,\ 2r-2s-B) & 2r & 2s \\
        \hline
      \end{array}
    \]
    \caption{Computing $K_B(r+s, r-s)$ in terms of $r$ and $s$ instead.}
    \label{tab:doubling}
  \end{table}

  In \Cref{tab:doubling}, we find that in every row except the first,
  the two projective classes are exactly $[2r]$ and $[2s]$, in some order.
  This means that \eqref{eq:conjugacy} is proved
  in every case except $A \ge D \ge E \ge C$.

  In the problematic case, we then note that $A \ge D$ and $E \ge C$ rearrange to
  \begin{align*}
    A \ge D &\iff d_1 \ge B - d_2 - 1 \iff 2r \ge B-1 \\
    E \ge C &\iff B - d_1 \ge d_2 - 1 \iff 2r \le B+1.
  \end{align*}
  So in fact this case can only occur when $2r = B-1$ or $2r=B+1$.
  In either case we have $[2r] = [1]$ in $\PB$ and \eqref{eq:conjugacy} is still true.
\end{proof}

\begin{remark}
  One may also simplify the casework by noting that every pair in $\TB$
  lies in the case treated by
  \cite[equation (1.1)]{HassePrichett} and \cite[\S2, ``type (a)'']{DevlinZeng},
  where it is shown that $K_B(d_1, d_2)$ is $\{|2d_1-B|, |2d_2-B|\}$ in some order.
  (In $\TB$, $d_1+d_2$ is even and hence never equals $B$.)
  This makes \Cref{prop:conjugacy} follow more directly,
  eliminating the need to compute \Cref{tab:doubling}.
\end{remark}

\begin{example}[The conjugacy in base $11$]\label{ex:conjugacy}
  Start from $(d_1,d_2)=(5,1)\in \mathcal{T}_{11}$.  Then
  \[
        r=\frac{5+1}{2}=3
        \qquad\text{and}\qquad
        s=\frac{5-1}{2}=2,
  \]
  so $\Phi(5,1)=\{[3],[2]\}$.  Doubling gives
  \[
        \{[3],[2]\}\longmapsto \{[6],[4]\}=\{[5],[4]\}
  \]
  in $P_{11}$, since $[6]=[-5]=[5]$.  The Kaprekar step gives
  \[
        (5,1)\longmapsto(9,1).
  \]
  For this new state,
  \[
        \Phi(9,1)=
        \left\{\left[\frac{10}{2}\right],
        \left[\frac{8}{2}\right]\right\}
        =\{[5],[4]\},
  \]
  exactly as Proposition~\ref{prop:conjugacy} predicts.
\end{example}

At this point \Cref{thm:structural} is immediate.
\begin{proof}[Proof of \Cref{thm:structural}]
  By \Cref{lem:transient}, every orbit reaches the forward-invariant region $\TB$
  after at most three steps in difference coordinates.
  By \Cref{lem:bijection},
  $\Phi$ is a bijection from this region to the unordered two-element subsets of $\PB$.
  By \Cref{prop:conjugacy}, this bijection carries one Kaprekar step to one doubling step.
  This proves both the pre-periodic statement and the asserted conjugacy on the stable region.
\end{proof}

\section{Cycle lengths}
\label{sec:cycle}
We have now replaced the Kaprekar map by projective doubling.  The rest of the argument is arithmetic.
The key quantity is the order of $2$ after signs are ignored
(see Yamagami-Matsui \cite{YamagamiMatsuiLoops};
denoted $\sigma(r)$ in Kay-Downes-Ward \cite{KayDownesWardOdd},
and \href{https://oeis.org/A003558}{OEIS A003558}).

\begin{definition}[Projective order of $2$]
  For odd $B>3$, define
  \begin{equation}\tag{L}\label{eq:lambda}
        \lambda(B)=\min\{m>0:2^m\equiv \pm1\pmod B\}.
  \end{equation}
  Equivalently, $\lambda(B)$ is the order of the class of $2$ in the finite group
  \[
        (\ZZ/B\ZZ)^\times/\{\pm1\}.
  \]
\end{definition}

This definition is tailored to $\PB$:
multiplying by $2^{\lambda(B)}$ brings every projective class back to itself.
Consequently the induced map on unordered pairs cannot have longer cycles.

\begin{lemma}\label{lem:orderbound}
  Every cycle of the doubling map on $\binom{\PB}{2}$ has length at most $\lambda(B)$.
\end{lemma}

\begin{proof}
  By definition, $2^{\lambda(B)}\equiv\pm1\pmod B$.
  Therefore multiplication by $2^{\lambda(B)}$ fixes every projective class $[x]\in \PB$,
  including classes represented by non-units.
  Hence the induced action on unordered two-element subsets of $\PB$ is also fixed
  by the $\lambda(B)$-th iterate.  Every orbit length is therefore at most $\lambda(B)$.
\end{proof}

\begin{lemma}[{\cite[Theorem 2]{KayDownesWardOdd}}]
  \label{lem:lambdabound}
  One has
  \[
        \lambda(B)\le \frac{\varphi(B)}{2}.
  \]
\end{lemma}

\begin{proof}
  The class of $2$ lies in the finite group $(\ZZ/B\ZZ)^\times/\{\pm1\}$.
  This group has order $\varphi(B)/2$, because $B$ is odd and $B>1$.
The order of an element in a finite group divides the order of the group.  Hence $\lambda(B)\le \varphi(B)/2$.
\end{proof}

This lets us prove part of Corollary~\ref{cor:length} now.
\begin{corollary}
  \label{cor:primeonly}
  If $B > 3$ is odd, then $c_{\max}(B) \le \frac{B-1}{2}$.
  Moreover, if equality holds then $B$ must be prime.
\end{corollary}
\begin{proof}
  By Theorem~\ref{thm:structural},
  every terminal Kaprekar cycle is represented by a cycle of the doubling map on $\binom{\PB}{2}$.
  Lemmas~\ref{lem:orderbound} and \ref{lem:lambdabound} give
  \[
        c_{\max}(B)
        \le \lambda(B)
        \le \frac{\varphi(B)}{2}
        \le \frac{B-1}{2}.
  \]
  This proves the stated bound.

  As for the equality case, suppose that $c_{\max}(B)=(B-1)/2$.  Then every inequality in
  \[
        c_{\max}(B)
        \le \lambda(B)
        \le \frac{\varphi(B)}{2}
        \le \frac{B-1}{2}
  \]
  must be an equality.  In particular, $\varphi(B)=B-1$.
  For an integer $B>1$, this is equivalent to $B$ being prime.
  Hence equality forces prime base.
\end{proof}

This does not say that every prime base is extremal.  It only says that composite bases are ruled out.
The examples below show both a composite base where the bound is strict and a
prime base where the bound is sharp.

\begin{example}[The order bound in bases $9$ and $11$]\label{ex:orderbound}
  For $B=9$, the projective order of $2$ is $3$, since $2^3\equiv -1\pmod 9$.
  Lemma~\ref{lem:orderbound} therefore forces every terminal four-digit cycle in
  base $9$ to have length at most $3$,
  and the examples in the introduction show that this bound is sharp in base $9$.

  For $B=11$, the projective order is $5$,
  since $2^5\equiv -1\pmod {11}$ and no smaller positive exponent gives $\pm1$.
  Thus the order bound permits cycles of length $5$,
  and the two length-five cycles displayed in the introduction attain it.
\end{example}

The same projective model also gives an explicit count of all terminal cycles.
We record this as a proposition because it is the most useful practical form of the classification.
For each odd divisor $n>1$ of $B$, put
\[
        \lambda(n)=\min\{m>0:2^m\equiv\pm1\pmod n\}.
\]
Following \cite[equation (5)]{KayDownesWardOdd}, for a positive integer $a$, define
\begin{equation}\label{eq:mu}
        \mu_B(a)=
        \sum_{\substack{n\mid B,\ n>1\\ \lambda(n)=a}}
        \frac{\varphi(n)}{2a}.
\end{equation}
Thus $\mu_B(a)$ is the number of $a$-cycles of the doubling map on $\PB$ itself,
before passing to unordered pairs.

\begin{theorem}[Explicit cycle counts]
  Let $N_B(L)$ be the number of terminal four-digit Kaprekar cycles of length $L$ in the odd base $B>3$,
  with constant digit strings omitted.  Then
  \begin{equation}\label{eq:cycle-count}
    \begin{split}
      N_B(L)
      ={}&
      \mu_B(L)\left\lfloor\frac{L-1}{2}\right\rfloor
      +\mu_B(2L)
      +\binom{\mu_B(L)}{2}L \\
      &\quad
      +\sum_{\substack{a<b\\ \operatorname{lcm}(a,b)=L}}
      \mu_B(a)\mu_B(b)\gcd(a,b),
    \end{split}
  \end{equation}
  where $\mu_B(a)$ is given by \eqref{eq:mu},
  and where $\mu_B(a)=0$ if no divisor $n$ contributes to \eqref{eq:mu}.
\end{theorem}

\begin{proof}
  We first compute the cycle structure of doubling on $\PB$;
  this goes along the lines of Kay-Downes-Ward \cite[\S2.2]{KayDownesWardEven}.
  Fix a nonzero residue class $[x]\in \PB$, and let $g=\gcd(x,B)$.  Write $B=gn$.
  Then $x=g y$ with $y$ a unit modulo $n$.
  Multiplication by $2$ preserves $g$,
  and the projective classes with this fixed value of $g$ are naturally identified with
  \[
        (\ZZ/n\ZZ)^\times/\{\pm1\}.
  \]
  On this set, multiplication by $2$ has order $\lambda(n)$.
  Hence all cycles in this stratum have length $\lambda(n)$, and the number of such cycles is
  \[
        \frac{|(\ZZ/n\ZZ)^\times/\{\pm1\}|}{\lambda(n)}
        =\frac{\varphi(n)}{2\lambda(n)}.
  \]
  Summing over all divisors $n>1$ of $B$ gives \eqref{eq:mu}.

  It remains to pass from cycles on $\PB$ to cycles on unordered two-element subsets of $\PB$.
  This is elementary cycle combinatorics.
  Suppose first that the two elements are chosen from two distinct cycles of lengths $a$ and $b$.
  If $a<b$, then each choice of the two cycles contributes $\gcd(a,b)$ orbits of unordered pairs,
  each of length $\operatorname{lcm}(a,b)$.  This gives the final sum in \eqref{eq:cycle-count}.
  If the two distinct cycles have the same length $L$,
  then each unordered pair of such cycles contributes $L$ orbits of length $L$,
  giving the term $\binom{\mu_B(L)}{2}L$.

  Finally suppose that both elements lie in the same $a$-cycle.  Label that cycle cyclically.
  An unordered pair is determined, up to rotation, by its cyclic distance
  \[
        1,2,\dots,\left\lfloor\frac a2\right\rfloor.
  \]
  For each distance smaller than $a/2$, the orbit has length $a$.  This contributes
  \[
        \mu_B(a)\left\lfloor\frac{a-1}{2}\right\rfloor
  \]
  orbits of length $a$.  If $a$ is even,
  the antipodal distance $a/2$ gives one additional orbit of length $a/2$ for each $a$-cycle;
  this accounts for the term $\mu_B(2L)$ in the count of length-$L$ orbits.

  By Theorem~\ref{thm:structural},
  terminal Kaprekar cycles are exactly these cycles of the doubling action on $\binom{\PB}{2}$.
  This proves the formula.
\end{proof}

\begin{corollary} The following are true.

\noindent
  (i) For odd $B>17$, there are at least two terminating $\lambda(B)$-cycles.

\smallskip
\noindent
  (ii) There is a unique terminating cycle for odd $B$ if and only if
  $B=3$, $B=5$ or $B=7$.
  (In other words, these are the only odd bases where
  all the non-constant strings have the same eventual behaviour.)
\end{corollary}
\begin{proof}
  Put $L=\lambda(B)$. Then \eqref{eq:mu} implies that $\mu_B(L)\geq 1$.
  If $B>17=2^4+1$, it is immediate from the definition of $\lambda$ that $L>4$,
  and thus the calculation of the first term of the right side \eqref{eq:cycle-count}
  implies the claim that $N_B(L)>1$, as all the other terms are non-negative.
  This implies (i).

  For (ii), put $N=\left|\PB\right|=(B-1)/2$, so that $\binom{\PB}{2}$ has $N(N-1)/2$ elements,
  while by Lemmas~\ref{lem:orderbound} and \ref{lem:lambdabound} every cycle of the
  doubling map has length at most $\lambda(B)\le \varphi(B)/2\le N$.
  If $B\geq 9$ then $N\geq 4$, hence $N(N-1)/2>N$ and the doubling action has at least
  two orbits. By Theorem~\ref{thm:structural}, there are at least two terminating cycles.
  Conversely $\binom{P_5}{2}$ has a single element, and for $B=7$ doubling permutes the
  three elements of $\binom{P_7}{2}$ in one $3$-cycle.
  Finally $\mathcal{T}_3=\varnothing$, so $B=3$ lies outside Theorem~\ref{thm:structural};
  a direct check shows that all five states of $X_3$ enter the cycle
  $(2,0)\mapsto(1,1)\mapsto(2,0)$.
  For $B=5,7$ this cycle is moreover shown to be unanimous in \cite{KayDownesWardOdd};
  for $B=3$ unanimity follows from the complete classification of \cite{YamagamiMatsuiLoops}.
\end{proof}

\begin{remark}
  Here the base $5$ corresponds to the unique terminating cycle being a $1$-cycle
  (fixed point, i.e.\ the original Kaprekar phenomena),
  and that happens \cite{HassePrichett} for even $B$,
  exactly when $B=2^{2n+1} \cdot 5$ for $n\geq 0$.
  The numerical evidence suggests that the even bases $B$ for which
  there is a unique terminating cycle are
  exactly $2^{2n+1} \cdot 3$, $2^{2n+1} \cdot 5$,  $2^{3n\pm 1} \cdot 7$,
  the cycle lengths being $6(n+1)$, $1$, and $3$ respectively.
\end{remark}

\section{Prime bases and the number of largest cycles}
\label{sec:prime}

The preceding section gives the full cycle-count formula and proves the universal bound.
In this section we finish the proofs of \Cref{cor:length} and \Cref{cor:primecriterion}.

\subsection{Proof of \texorpdfstring{\Cref{cor:length}}{Corollary \ref{cor:length}}}
Recalling \Cref{cor:primeonly},
we suppose $B = p > 3$ is prime and set $q \coloneq (p-1)/2$.
The simplification is that every nonzero residue is a unit,
so $P_p$ is a group quotient rather than merely a set of projective residue classes.
Thus the projective set $P_p=\FF_p^\times/\{\pm1\}$ has exactly $q$ elements.
The doubling map on $P_p$ has order
\[
        \lambda(p)=\min\{m>0:2^m\equiv \pm1\pmod p\}.
\]

If $p>5$, then $q\ge3$.  In this case,
the induced action on unordered two-element subsets has a cycle of length $q$ if
and only if multiplication by $2$ is a single $q$-cycle on $P_p$.
Equivalently, this happens if and only if $\lambda(p)=q$.  Therefore, for primes $p>5$,
\[
        c_{\max}(p)=\frac{p-1}{2}
        \quad\Longleftrightarrow\quad
        \lambda(p)=\frac{p-1}{2}.
\]
This proves the equality criterion in Corollary~\ref{cor:length}.

The prime $p=5$ is exceptional for a simple reason.
Here $q=2$, so $P_5$ has only two elements and $\binom{P_5}{2}$ has only one element.
Hence the largest induced cycle has length $1$, even though $\lambda(5)=2$.

\subsection{Proof of \texorpdfstring{\Cref{cor:primecriterion}}{Corollary \ref{cor:primecriterion}}}
It remains only to count the largest cycles in the equality case.
This is a small piece of cycle combinatorics:
unordered pairs on a single cycle are classified by their cyclic distance.

Assume $p>5$ and $\lambda(p)=q$.  Then doubling is a single $q$-cycle on $P_p$.
The induced action on unordered two-element subsets of a $q$-cycle is classified by cyclic distance.
The possible distances are
\[
        1,2,\dots,\left\lfloor\frac q2\right\rfloor.
\]
For a fixed distance $j$, all unordered pairs at distance $j$ form one orbit under cyclic rotation.
If $q$ is odd, every such orbit has length $q$.
If $q$ is even, the distance $q/2$ orbit has length $q/2$, not $q$, because each pair is antipodal.
Therefore the number of orbits of length $q$ is
\[
        \begin{cases}
        (q-1)/2, & q\text{ odd},\\[4pt]
        q/2-1, & q\text{ even}.
        \end{cases}
\]
Both cases are equal to
\[
        \left\lfloor\frac{q-1}{2}\right\rfloor.
\]
Since $c_{\max}=q$, the number of largest terminal Kaprekar cycles is
\[
        \left\lfloor\frac{c_{\max}-1}{2}\right\rfloor.
\]
This completes the proof of Corollary~\ref{cor:primecriterion}.

\begin{example}[Counting largest cycles when $p=11$]\label{ex:counting}
  For $p=11$, the projective set $P_{11}$ has $q=5$ elements.
  Since $2$ has projective order $5$, doubling is a single five-cycle on $P_{11}$.
The unordered pairs of elements of a five-cycle have two possible cyclic distances, namely $1$ and $2$.
These two distances give the two length-five Kaprekar cycles displayed in the introduction.  This is the count
  \[
        \left\lfloor\frac{5-1}{2}\right\rfloor=2.
  \]
For $p=7$, the projective set has $q=3$ elements and only one cyclic distance, so there is one largest cycle,
  again as in the introduction.
\end{example}

\section{Appendix: AxiomProver's autonomous Lean verification}\label{sec:AI}
Here we provide the context for this project as well as the protocol used for Lean
formalization and verification (see \cite{Mathlib2020, Lean}).

\subsection*{Process}
The formal proofs provided in this work were developed and verified using Lean~4.28.0.
Compatibility with earlier or later versions is not guaranteed due to the
evolving nature of the Lean 4 compiler and its core libraries.
The relevant files are all posted in the following repository:
\begin{center}
  \url{https://github.com/AxiomMath/kaprekar4}
\end{center}
The input files were:
\begin{itemize}
  \item \texttt{kaprekar-input.tex}: statements
    of \Cref{thm:structural,cor:length,cor:primecriterion}
    (and the definitions necessary to state these main results)
  \item \texttt{task.md}: short instructions referring to the Kaprekar input \TeX
  \item \texttt{.environment}: specifies Lean version to use
\end{itemize}
We note that we did not provide the proofs of the theorems to AxiomProver, only the statements.
Based on these input files, AxiomProver generated two output files:
\begin{itemize}
  \item \texttt{problem.lean}, a Lean 4.28.0 formalization of the problem statement; and
  \item \texttt{solution.lean}, a complete Lean 4.28.0 formalization of the proof.
\end{itemize}

With the formal solutions for guidance, the human authors wrote this paper
(without the use of AI) for human readers.
At first glance, the proofs found by AxiomProver may not resemble the narrative presented in this paper.
Turning a Lean file into a human-readable proof is difficult
because Lean is written as code for a type-checker.

\section*{Declaration of generative AI and AI-assisted technologies in the manuscript preparation process}
As described in the preceding Appendix,
AxiomProver (an AI tool under development)
was used to produce a formal proof of \Cref{thm:structural,cor:length,cor:primecriterion}.
The paper was written without AI.


\begin{thebibliography}{99}

  \bibitem{DevlinZeng}
  P. Devlin and T. Zeng,
  \emph{Maximum distances in the four-digit Kaprekar process},
  Integers \textbf{21} (2021), Paper No. A97.

\bibitem{Gardner}
  M. Gardner,
  \emph{} Mathematical Games, pa.\ 112, March 1975,
\url{https://www.scientificamerican.com/article/mathematical-games-1975-03/}

  \bibitem{HassePrichett}
  H. Hasse and G. D. Prichett,
  \emph{The determination of all four-digit Kaprekar constants},
  Journal f\"ur die reine und angewandte Mathematik \textbf{299--300} (1978), 113--124.

  \bibitem{Jordan}
  J. H. Jordan,
  \emph{Self-producing sequences of digits},
  American Mathematical Monthly \textbf{71} (1964), 61--64.

  \bibitem{Kaprekar1949}
  D. R. Kaprekar,
  \emph{Another solitaire game},
  Scripta Mathematica \textbf{15} (1949), 244--245.

  \bibitem{Kaprekar1955}
  D. R. Kaprekar,
  \emph{An interesting property of the number 6174},
  Scripta Mathematica \textbf{21} (1955), 304.

  \bibitem{KayDownesWardOdd}
  A. Kay and K. Downes-Ward,
  \emph{Fixed points and cycles of the Kaprekar transformation: 1. Odd bases},
  Journal of Integer Sequences \textbf{25} (2022), Article 22.6.7.

  \bibitem{KayDownesWardEven}
  A. Kay and K. Downes-Ward,
  \emph{Fixed points and cycles of the Kaprekar transformation: 2. Even bases},
  arXiv:2408.12257.

  \bibitem{Mathlib2020}
  The mathlib Community,
  The {L}ean mathematical library,
  in \emph{Proceedings of the 9th ACM SIGPLAN International Conference on
  Certified Programs and Proofs (CPP 2020)},
  ACM, 2020.

  \bibitem{Lean}
  L.~de~Moura, S.~Kong, J.~Avigad, F.~van~Doorn, and J.~von~Raumer,
  The {L}ean theorem prover (system description),
  in \emph{Automated Deduction -- CADE-25},
  Lecture Notes in Computer Science~9195, Springer, 2015, 378--388.

  \bibitem{Prichett}
  G. D. Prichett,
  \emph{Terminating cycles for iterated difference values of five digit integers},
  Journal f\"ur die reine und angewandte Mathematik \textbf{303--304} (1978), 379--388.

  \bibitem{Thakur}
  D. S. Thakur,
  \emph{Kaprekar phenomena},
  Proceedings of Ropar Conference, RMS Lecture Notes Series \textbf{26} (2019), 1--10.

  \bibitem{YamagamiMatsuiLoops}
  A. Yamagami and Y. Matsui,
  \emph{On 3-adic Kaprekar loops},
  JP Journal of Algebra, Number Theory and Applications \textbf{40} (2018), 957--1028.

  \bibitem{YamagamiMatsuiConstants}
  A. Yamagami and Y. Matsui,
  \emph{On some formulas for Kaprekar constants},
  Symmetry \textbf{11} (2019), Article 885.

\end{thebibliography}
\end{document}